\documentclass{amsart}

\usepackage{amssymb,amsmath,amsfonts,array,delarray,amsthm,graphics,graphicx}

\def\C {\mathbb C}
\newcommand{\abs}[1]{| #1 |}
\newcommand{\Abs}[1]{\left| #1\right|}
\newcommand{\B}{\mathbb B}
\def\bcases{\begin{cases}}
\newcommand{\bk}[1]{\left[#1\right]}
\newcommand{\br}[1]{\left(#1\right)}
\def\ecases{\end{cases}}
\newcommand{\cl} {\overline}

\newcommand{\dd}{\delta}
\newcommand{\D}{\mathbb D}
\def\dist{\operatorname{dist}}
\newcommand{\e}{\epsilon}
\newcommand{\Frac}{\displaystyle\frac}

\newcommand{\norm} [1]{\left\| #1\right\|}
\newcommand{\of}{\circ}
\newcommand{\p}{\partial}

\newcommand{\re}{\text{\rm Re}\,}
\newcommand{\set}[1]{\left\{ #1\right\}}
\def\sm{\setminus}

\newcommand{\To}{\longrightarrow}

\newcommand{\W}{\Omega}
\newcommand{\z}{\zeta}

\newtheorem{thm}{Theorem}
\newtheorem{cor}{Corollary}
\newtheorem{lemma}{Lemma}
\newtheorem{prop}{Proposition}

\theoremstyle{definition}

\theoremstyle{remark}
\newtheorem{remark}{Remark}

\newcommand{\be}{\begin{equation}}
\newcommand{\ee}{\end{equation}}

\title{Kobayashi, Carath\'eodory,
 and Sibony metric}
      
      \author{John Erik Forn\ae ss, Lina Lee}
     \thanks{The first author is supported by NSF grant} 
\begin{document}
\maketitle

\begin{abstract}
In this paper, we estimate the boundary behaviour of the Sibony metric near a pseudoconcave boundary point. We show that the metric blows up at a different rate than the Kobayashi metric. 
 \end{abstract}
 
\section{Introduction}
Invariant metrics play an important role in Complex Analysis. Yet many of their basic properties are still unknown. In this paper the authors investigate the boundary behaviour of the Kobayashi, Carath\'eodory, and Sibony metrics near pseudoconcave boundary points. We show that their growth rates are different. So we work near a boundary point where at least one of the eigenvalues of the Levi form is strictly negative. We let $F_K, F_S, F_C$ denote the Kobayahsi, Sibony and Carath\'eodory metrics respectively.
Our main result is the following:

\begin{thm}\label{333}
Let $\W$ be a bounded domain in $\C^n, n>1,$ with $C^2$-boundary.
Let $P$ be a boundary point which is not pseudoconvex. Let $P_\delta$ be the point on the inner normal to $p$ at distance $\delta$ and let $\nu$ be a unit complex normal vector to $\partial \W$ at $P.$
Then $F_K^\W(P_\dd,\nu) \approx  \frac{1}{\dd^{3/4}}, F_S^\W(P_\dd,\nu) \approx  \frac{1}{\dd^{1/2}},
 F_C^\W(P_\dd,\nu) \approx  1.$
\end{thm}

The main new result in this theorem is the estimate for the Sibony metric, which is the following: 
\begin{thm}\label{332}
Let $\W=\set{1/4<\abs z^2+\abs w^m<1}\subset\C^2$, $m\ge 2$ and $P_\dd=(1/2+\dd,0)$. Then
$$
F_S^\W(P_\dd,\nu)\approx\frac{1}{\dd^{1-\frac{1}{m}}},\quad\nu=(1,0),
$$
for $\dd>0$ small enough.
\end{thm}

The estimate for the Kobayahsi metric, Theorem \ref{645}, is due to Krantz, see \cite{krantz-annulus}: $F_K^\W(P_\dd,\nu)\approx\frac{1}{\dd^{1-1/(2m)}}$.

In \cite{fornaess}, it is shown that the Kobayashi metric and the Sibony metric are different on a ring domain. Theorem \ref{332} shows that the Sibony metric has actually a smaller blowing up rate than the Kobayashi metric does.  

The plan of the paper is as follows. In the second section, we give 
background information. In the third section, we prove Theorem \ref{332}. Theorem \ref{333} is proved in Section 4.

We would like to thank Prof. Steven G. Krantz for suggesting this problem. 
\section{Background}
In this section, we give definitions and properties of the metrics, which are used in later sections. For more detailed discussion of the metrics, refer \cite{Kobayashi}, \cite{Krantz}, \cite{Royden}, and \cite{sibony}.

Let $\W\subset\C^n$ be a domain, $P\in\W$, and $\xi=(\xi_1,\dots,\xi_n)\in\C^n$. The Kobayashi metric $F_K^\W(P,\xi)$, the Carath\'eodory metric $F_C^\W(P,\xi)$, and the Sibony metric $F_S^\W(P,\xi)$ are defined as follows:
\begin{align}
F_K^\W(P,\xi)&=\inf\set{\alpha:\exists\phi\in\W(\D), \text{ s.t. }\phi(0)=P, \,\phi'(0)=\xi/\alpha, \, \alpha>0}\\
F_C^\W(P,\xi)&=\sup\set{\abs{f_*(P)\xi}=\Abs{\sum_{i=1}^n\frac{\p f(P)}{\p z_i}\xi_i}: f\in\D(\W),\, f(P)=0}\\
\label{648}
F_S^\W(P,\xi)&=\sup\set{\br{\p\cl\p u(P)(\xi,\cl\xi)}^{1/2}=\br{\sum_{i,j=1}^n\frac{\p^2 u(P)}{\p z_i\p\cl z_j}\xi_i\cl\xi_j}^{1/2}: u\in A_\W(P)},
\end{align}
where $\D$ denotes a unit disc in $\C$, $\W_2(\W_1)$ the family of holomorphic mappings from $\W_1$ to $\W_2$ and $A(P,\W)$ is the set of plurisubharmonic functions on $\W$ such that $u\in A(P,\W)$ if $u(P)=0$, $u$ is $C^2$ near $P$, $\log u$ is plurisubharmonic on $\W$, and $0\le u\le 1$ on $\W$. 

The three metrics satisfy the non-increasing property under holomorphic mappings. 

\begin{lemma}
Let $\W_1$ and $\W_2$ be domains in $\C^n$ and $\C^m$ respectively, $\Phi:\W_1\To\W_2$ be a holomorphic mapping, $P\in\W_1$, and $\xi\in T_P(\W_1)$. Then we have
\be\label{623}
F^{\W_1}(P,\xi)\ge F^{\W_2}(\Phi(P),\Phi_*(P)\xi)
\ee
\end{lemma}
\begin{cor}\label{cor}
If $\W_1\subset\W_2$, then
\be\label{443}
F^{\W_1}(P,\xi)\ge F^{\W_2}(P,\xi).
\ee
\end{cor}
The Kobayashi metric and the Carath\'eodory metric are two extremes of the pseudo-metrics in the sense that if $F$ is a pseudometric defined in such a way that it coincides with the Poincar\'e metric on the unit disc in $\C$ and satisfies (\ref{623}), then we have
$$
F_C^\W(P,\xi)\le F^\W(P,\xi)\le F_K^\W(P,\xi),
$$
on any domain $\W$. 
The Sibony metric coincides with the Poincar\'e metric on the unit disc and satisfies the non-increasing property under holomorphic mappings \cite{sibony}. Hence we have
\be\label{1054}
F_C^\W(P,\xi)\le F_S^\W(P,\xi)\le F_K^\W(P,\xi).
\ee
The Sibony metric is defined as the supremum of the Hessian of certain plurisubharmonic functions (\ref{648}). Hence we have the following property, which is stated in \cite{sibony} without a proof. We include the proof for the convenience of the reader.
\begin{lemma}\label{1054}
$$
F_S^\W(P,\xi_1+\xi_2)\le F_S^\W(P,\xi_1)+F_S^\W(P,\xi_2)
$$
\end{lemma}
\begin{proof}
\begin{multline*}
\br{F_S^\W(P,\xi_1+\xi_2)}^2=\sup_{u\in A_\W(P)}\p\cl\p u(\xi_1+\xi_2,\cl\xi_1+\cl\xi_2)\\
=\sup_u\bk{\p\cl\p u(\xi_1,\cl\xi_1)+\p\cl\p u(\xi_2,\cl\xi_2)+2\re\p\cl\p u(\xi_1,\cl\xi_2)}\\
\le \sup_u\bk{\p\cl\p u(\xi_1,\cl\xi_1)+\p\cl\p u(\xi_2,\cl\xi_2)+2\br{\p\cl\p u(\xi_1,\cl\xi_1)}^{1/2}\br{\p\cl\p u(\xi_2,\cl\xi_2)}^{1/2}}\\
=\sup_u \bk{\br{\p\cl\p u(\xi_1,\cl\xi_1)}^{1/2}+\br{\p\cl\p u(\xi_2,\cl\xi_2)}^{1/2}}^2\\
\le \bk{\br{\sup_u \p\cl\p u(\xi_1,\cl\xi_1)}^{1/2}+\br{\sup_u \p\cl\p u(\xi_2,\cl\xi_2)}^{1/2}}^2
=\br{F_S^\W(P,\xi_1)+F_S^\W(P,\xi_2)}^2
\end{multline*}
\end{proof}

In \cite{krantz-annulus}, Krantz shows the asymptotic behavior of the Kobayashi metric near the inner boundary of an annulus in the normal direction. 

Let
$$
\W=\set{(z,w)\in\C^2:\frac{1}{4}<\abs z^2+\abs w^m<1, m \geq 2}
$$
and $P_\dd=(p,0)=(1/2+\dd,0)$ and $\nu=(1,0)$. 
\begin{thm}[Krantz \cite{krantz-annulus}]\label{645}
$$
F_K^\W(P_\dd,\nu)\approx\br{\frac{1}{\dd}}^{1-\frac{1}{2m}}.
$$
\end{thm}

Since the holomorphic convex hull of $\W$ is $\B=\set{\abs z^2+\abs w^2\le 1}\subset\C^2$, The Carath\'eodory metric on $\W$ coincides with the Carath\'eodory metric on $\B$, which we can explicitly calculate using the M\"obius tranform of $\B$:
\begin{prop}
$$
F_C^\W(P_\dd,\nu)\approx 1.
$$
\end{prop}
\begin{proof}
Let $\Phi$ be the M\"obius transform of $\B$ that maps $P_\dd$ to $0$:
$$
\Phi=\br{\frac{z-p}{1-pz},\frac{\sqrt{1-p^2}w}{1-pw}}
$$
Hence we get
\be\label{157}
F_C^\B(P_\dd,\nu)=F_K^\B(P_\dd,\nu)
=\abs{\Phi_*(P_\dd) \nu}
=\Abs{\frac{1}{1-p^2}}
\ee
\end{proof}

In section 4, we use the localization of the Kobayashi metric and the Sibony metric. We prove the localization of the Sibony metric in section 4 and here we present a proof of the localization of the Kobayashi metric. It first appeared in \cite{Royden} and later in \cite{Graham}.
\begin{lemma}\label{427}
Let $\W\subset\C^n$, $P\in\cl\W$, and $U$ be a neighborhood of $P$. If  $V\subset\subset U$ and $P\in V$, then we have
$$
F_K^\W(q,\xi)\approx F_K^{\W\cap U}(q,\xi)\quad\forall q\in \W\cap V, \,\xi\in\C^n.
$$
\end{lemma}
\begin{proof}
Let $r$ be such that
$$
r=\inf\set{a>0: \exists \phi\in\W(\D),\, \phi(0)=p,\,\phi(a)=p',\, \text{for some }p\in V\cap\W,\,p'\in \W\sm U}.
$$
If $f\in\W(\D)$ satisfies $f(0)=q$ and $f'(0)=\xi/\alpha$, then $g(\z):=f(r\z)\in\W\cap U(\D)$. Hence 
$$
F_K^{\W\cap U}(q,\xi)\le\frac{1}{r} F_K^\W(q,\xi).
$$ 
\end{proof}
\begin{remark}
The localization of the Carath\'eodory metric was proved in \cite{Graham} on a strongly pseudoconvex domain using the existence of the peak function. Theorem \ref{333} does not require the localization of the Carath\'eodory metric. 
\end{remark}

\section{Estimation of the Sibony metric on an Annulus}

Throughout this section, we let $\W=\B\sm\set{(z,w)\in\C^2:\abs z^2+\abs w^m\le1/4}$ ($m\ge 2$), where $\B$ is a unit ball in $\C^2$ with center $0$, $P_\dd=(p,0)=(1/2+\dd,0)$, and $\nu=(1,0)$.

\begin{lemma}\label{1015}
Let $P\in\W$ and $\xi\in\C^2$ be such that the complex line $\phi(\z)=P+\xi\z$ does not touch the inner boundary of $\W$ for all $\z\in\C$. Then $F_K^\W(P,\xi)=F_S^\W(P,\xi)=F_C^\W(P,\xi)=F_C^\B(P,\xi)$. 
\end{lemma}
\begin{proof}
By (4) and (6), it is enough to show that $F_K^\W(P,\xi)=F_C^\B(P,\xi)$. 

 Since $\W\subset\B$, we have $F_K^\W(P,\xi)\ge F_K^\B(P,\xi)= F_C^\B(P,\xi)$. Hence it is enough to show $F_K^\W(P,\xi)\le F_C^\B(P,\xi)$. 

Let $\Delta=\W\cap\set{\phi(\z):\z\in\C}$.  We have $F_K^\Delta(P,\xi)\ge F_K^\W(P,\xi)$, since $\Delta\subset\W$. We show that $F_K^\Delta(P,\xi)\le F_C^\B(P,\xi)$ by finding a holomorphic mapping $\psi:\B\To \Delta$ such that $\psi(P)=P$ and $\psi_*(P)\xi=\xi$: Let $\psi=f^{-1}\of\pi\of f$, where $f$ is a M\"obius transformation of $\B$ that sends $P$ to $0$ and $\pi$ is the projection of $\B$ onto $f(\Delta)$.
\end{proof}
\begin{prop} 
\be\label{401}
F_S^\W(P_\dd,\nu)\lesssim\frac{1}{\dd^{1-1/m}}
\ee
\end{prop}

\begin{proof}

Let $\beta>0$ be such that the complex line $\phi(\z)=P_\dd+\z(1,v)$, $\z\in\C$, does not touch the inner boundary if $\abs v>\beta$. We may write $\nu=(1,0)=(1/2,v)+(1/2,-v)$ for some $v\in\C$ such that $\abs v>\beta$. By Lemma \ref{1054}, Lemma \ref{1015}, and (\ref{157}), we get
\begin{align*}
F_S^\W(P_\dd,\nu)&\le F_S^\W(P_\dd, (1/2,v))+F_S^\W(P_\dd, (1/2,-v))\\
&=F_C^\B(P_\dd,(1/2,v))+F_C^\B(P_\dd,(1/2,-v))\\
&=\frac{2}{1-p^2}\bk{\frac{1}{4}+(1-p^2)\abs v^2}^{1/2}
\end{align*}

Since the above inequality holds for all $v\in\C$ such that $\abs v>\beta$, we have
\be\label{400}
F_S^\W(P_\dd,\nu)\le \frac{2}{1-p^2}\bk{\frac{1}{4}+(1-p^2)\beta^2}^{1/2}.
\ee

Now we estimate $\beta$: Since $\phi(\z)$ does not touch the inner boundary if $\abs v>\beta$, we have
\be\label{339}
\frac{1}{4}<\norm{\phi(\z)}^2=\Abs{\frac{1}{2}+\dd+\z}^2+\abs{\z v}^m,\quad\forall \z\in\C \text{ if }\abs v>\beta.
\ee
Since we have a lower bound on the right hand side of (\ref{339}) as follows:
$$
\frac{1}{4}+\re(\dd+\z)+\abs{\dd+\z}^2+\abs \z^m\abs v^m\ge
\frac{1}{4}+\dd-\abs\z+\abs v^m\abs\z^m,
$$

\noindent we can estimate $\beta$ from above by finding the condition on $\abs v$ such that $f(x)=1/4+\dd-x+\abs v^m x^m>1/4$ for all $x\ge 0$. Since $f$ has only one critical point on the positive $x$-axis, which is $(\abs v^m m)^{-1/(m-1)}$, it is equivalent to finding the condition on $\abs v$ such that
$$
f((\abs v^m m)^{-1/(m-1)})=\frac{1}{4}+\dd- (\abs v^m m)^{-1/(m-1)}+\abs v^m (\abs v^m m)^{-m/(m-1)}\ge \frac{1}{4}.
$$

Hence we get
$$
\dd\ge \abs v^{-m/(m-1)}\br{ m^{-1/(m-1)}-m^{-m/(m-1)}}
$$
and therefore
$$
\abs v\ge C\frac{1}{\dd^{1-1/m}},
$$ 
where $C$ is a constant depending on $m$. Therefore 
\be\label{359}
\beta\le C\frac{1}{\dd^{1-1/m}}
\ee
and (\ref{359}) together with (\ref{400}) proves (\ref{401}). 
\end{proof}
\begin{remark}
For a general tangent vector $\nu=(a,b)=a(1,0)+b(0,1)= N+T$ we get the inequality
$F_S^\W(P_\dd,\nu)\lesssim\frac{1}{\dd^{1-1/m}}|N|+ |T|$
\end{remark}
\begin{remark}
If $\W=\set{1/4<\abs {z_1}^2+\abs{z_2}^{m_2}+\cdots+\abs{z_n}^{m_n}<1}\subset\C^n$, $2\le m_2\le m_3\le\cdots\le m_n$, $P_\dd=(1/2+\dd,0,\dots,0)$ and $\nu=(1,0,\dots,0)$, then
$$
F_S^\W(P_\dd,\nu)\lesssim \frac{1}{\dd^{1-\frac{1}{m_2}}}.
$$
It is because the metric on $\W$ is less than the metric on the slice of $\W$: $\W'=\set{1/4<\abs{z_1}^2+\abs{z_2}^{m_2}<1}\cap\set{z_j=0, \, j\ge 3}$.
\end{remark}
\begin{prop}\label{400}
$$
F_S^\W(P_\dd,\nu)\gtrsim\frac{1}{\dd^{1-1/m}}
$$
\end{prop}
\begin{proof}
To construct a function $u(z,w)$ giving a lower bound on the Sibony metric, we first find a function
which is a candidate for small $|w|$ and then patch with a function of $|w|$ to globalize. For small $|w|$
we take advantage of the fact that $|z|>1/2-\delta/2$ to get a function with large derivative in the $z$ direction. We give next the precise construction:\\
Let
$$
f(z)=\dd^{\frac{2}{m}}\Abs{\frac{z-p}{z-p+2\dd}}^2=\dd^{\frac{2}{m}}\Abs{\frac{z-1/2-\dd}{z-1/2+\dd}}^2
$$
and we define a plurisubharmonic function $u(z,w)$ on $\W$ as follows:
$$
u(z,w)=
\bcases
\max\set{\log\br{f(z)+\abs w^2}, \log\br{L\abs w^{2+\e}}}-L',& \abs w<c^{2/m}\dd^{1/m}\\
\log\br{L\abs w^{2+\e}}-L', & \abs w\ge c^{2/m}\dd^{1/m},
\ecases
$$
for some small constant $c\in (0,1/2)$, $0<\e<<1$, and large constants $L, L'$, which will be defined later. 

We shall show that $u(z,w)<0$ on $\W$,  that $\exp u(z,0)=e^{-L'}f(z)$ near $P_\dd$ and that $u$ is plurisubharmonic. 

Let $\W':=\W\cap\set{\abs w<c^{2/m}\dd^{1/m}}$. We shall show that $f(z)\le C\dd^{2/m}$ for some constant $C$ on $\W'$. Since $\W'\subset \set{ 1/4-c^2\dd<\abs z^2<1,\, \abs w<c^{2/m}\dd^{1/m}}$, it is enough to show that $f(z)\le C\dd$ for all $z$ such that $1/4-c^2\dd<\abs z^2<1$.

Let $z=x+iy$. Then we have
$$
f(z)=\dd^{\frac{2}{m}}\frac{(x-1/2-\dd)^2+y^2}{(x-1/2+\dd)^2+y^2}
$$
Hence $f(z)\leq\dd^{\frac{2}{m}}$ if $x-1/2\ge 0$. 

If $x-1/2\le 0$, then $(x-1/2+\dd)^2+y^2\le (x-1/2-\dd)^2+y^2$. Since $A/B\ge (A+C)/(B+C)$, if $A\ge B>0$ and $C\ge 0$, we have
$$
f(z)\le\dd^{2/m}\frac{(x-1/2-\dd)^2}{(x-1/2+\dd)^2}, \quad \text{if }x\in \bk{-1,-\sqrt{\frac{1}{4}-c^2\dd}}\cup\bk{\sqrt{\frac{1}{4}-c^2\dd},\frac{1}{2}}
$$
and
$$
f(z)\le\dd^{2/m}\frac{(x-1/2-\dd)^2+1/4-c^2\dd-x^2}{(x-1/2+\dd)^2+1/4-c^2\dd-x^2},\quad\text{if }x\in\bk{-\sqrt{\frac{1}{4}-c^2\dd},\sqrt{\frac{1}{4}-c^2\dd}}.
$$
A simple calculation shows that 
\be\label{1241}
f(z)\le \dd^{\frac{2}{m}}\frac{(1/2+\dd-\sqrt{1/4-c^2\dd})^2}{(1/2-\dd-\sqrt{1/4-c^2\dd})^2}
\le\dd^{\frac{2}{m}}\frac{(1+3c^2)^2}{(1-3c^2)^2}\le \dd^{\frac{2}{m}}(1+3c^2)^2(1+5c^2)^2,
\ee
for $c\le 1/3$, since
\be\label{1242}
\frac{1}{2}-3c^2\dd<\sqrt{\frac{1}{4}-c^2\dd}<\frac{1}{2}-c^2\dd.
\ee

Therefore 
$f(z)\le 5\dd^{2/m}$ for all $(z,w)\in\W\cap\set{\abs w<(\dd/3^2)^{1/m}}.$

Hence if we let $L=200$ and $c=1/3$, then
$$
f(z)+\abs w^2\le 5\dd^{2/m}+\dd^{2/m}/9\le 6\dd^{\frac{2}{m}} ,\quad \br{\frac{\dd}{4^2}}^{1/m}<\abs w<\br{\frac{\dd}{3^2}}^{1/m}
$$
and 
$$
6\dd^{\frac{2}{m}}<100\frac{\dd^{\frac{2}{m}+\frac{\e}{m}}}{16} <200\abs w^{2+\e},\quad\abs w>\br{\frac{\dd}{4^2}}^{\frac{1}{m}},
$$
for $\e<\frac{m(\log(0.96))}{\log\dd}$. 
Therefore
$$
\max\set{\log(f(z)+\abs w^2), \log(200\abs w^{2+\e})}=\log (200\abs w^{2+\e}),\quad \br{\frac{\dd}{4^2}}^{1/m}<\abs w<\br{\frac{\dd}{3^2}}^{1/m}
$$
Hence the function $u(z,w)$ is a well-defined plurisubharmonic function. Since $\abs w<1$ on $\W$, we may choose $L'=\log(200)$. 

Now we look at a small neighborhood of $P_\dd$. 
If $|w|<\left(\frac{1}{200}\right)^{\frac{1}{\e}}$, then $200|w|^{2+\e}<|w|^2$ so
$\log 200|w|^{2+\e}<\log (f(z)+|w|^2).$ Hence $e^u$ is smooth in a neighborhood of $P_\delta$, $e^{u(z,w)}=e^{-L'}f(z).$ 
$$
\br{\p\cl\p e^u(P_\dd)(\nu,\cl \nu)}=\frac{e^{-L'}}{4\dd^{2-\frac{2}{m}}}
$$
Hence the lower bound for the Sibony metric follows.
\end{proof}
\begin{remark}
For a general tangent vector $\nu=(a,b)=a(1,0)+b(0,1)= N+T$ we get the inequality
$F_S^\W(P_\dd,\nu)\gtrsim\frac{1}{\dd^{1-1/m}}|N|+ |T|$
\end{remark}
\begin{remark}
If $\W=\set{1/4<\abs {z_1}^2+\abs{z_2}^{m_2}+\cdots+\abs{z_n}^{m_n}<1}\subset\C^n$, $2\le m_2\le m_3\le\cdots\le m_n$, $P_\dd=(1/2+\dd,0,\dots,0)$ and $\nu=(1,0,\dots,0)$, then
\be\label{349}
F_S^\W(P_\dd,\nu)\gtrsim \frac{1}{\dd^{1-\frac{1}{m_2}}}.
\ee
Since $\abs{z_1}^2+\abs{z_2}^{m_2}+\cdots+\abs{z_n}^{m_n}\le \abs{z_1}^2+\abs{z_2}^{m_2}+\cdots+\abs{z_n}^{m_2}\le \abs{z_1}^2+\abs{z'}^{m_2}$, $z'=(z_2,\dots, z_n)$, we have $\W\subset\set{1/4<\abs{z_1}^2+\abs{z'}^{m_2}}\cap\set{\abs{z_1}^2+\abs{z_2}^{m_2}+\cdots+\abs{z_n}^{m_n}<1}$. The argument of Proposition \ref{400} goes through with $\abs{z'}$ in place of $\abs w$. 
\end{remark}
\section{Comparison of the metrics on a general domain}

We first prove a localization Lemma for the Sibony metric. 
\begin{lemma}\label{426}
Let $\W\subset\C^n$ be a bounded domain. If $V\subset\subset U$ are open sets, then we have
$$
F_S^{U\cap\W}(q,\xi)\approx F_S^\W(q,\xi),\quad\forall q\in V\cap\W, \,\forall\xi\in\C^n
$$
\end{lemma}
\begin{proof}
Since $U\cap\W\subset\W$, by (\ref{443}), we have $F_S^{U\cap\W}(z,\xi)\ge F_S^\W(z,\xi)$. 

Now we show the other direction. Let 
$$
r:=\dist(V\cap \W, \W\setminus U).
$$
For $q\in V\cap \W$, if $u\in A(q,U\cap\W)$, then define 
$$
v:=\bcases
\max\set{ \log \br{u+\e\abs{z-q}^2}, \log\Frac{2\abs{z-q}^4}{r^4}}-L,& z\in B(q,r)\cap\W\\
\log\Frac{2\abs{z-q}^4}{r^4}-L,\quad z\in\W\sm B(q,r)
\ecases
$$
where $\e>0$ is a very small constant such that $\e\abs{z-q}^2\le 1/2$ for all $z\in\W$, and $L$ is a large constant that will be chosen later. 

First we show that $v=\log(u+\e\abs{z-q}^2)-L$ near $q$: Let $\abs{z-q}=\dd$. Since $u
\ge 0$, we have $\log(u+\e\abs{z-q}^2)\ge \log\e+2\log\dd$. We also have
$$
\log \frac{2\abs{z-q}^4}{r^4}=\log 2+4\log\dd -4\log r
$$
Hence for $\dd$ small enough, $v=\log(u+\e\abs{z-q}^2)-L$ for all $z$ such that $\abs{z-q}\le\dd$. 

Now we choose $L$ such that $v\le 0$ on $\W$: Since $u+\e\abs{z-q}^2\le 3/2$ for all $z\in U\cap\W$, we have
$$
v=\log\frac{2\abs{z-q}^4}{r^4}-L, \quad z\in U\cap\W\cap\set{\abs{z-q}^4>3r^4/4}.
$$
Hence $v$ is a well defined plurisubharmonic function and for a large constant $L$, we have $v\le 0$ on $\W$.

\end{proof}



\noindent{\bf Proof of Theorem \ref{333}}

\begin{proof}
Since $\W$ has a $C^2$-boundary, we can find a small ball of radius $r$, $B_r$, that lies outside $\W$ and that is tangent to $\p\W$ at $P$ and a small neighborhood $U$ of $P$ such that $\W\cap U\subset U\sm B_r$. Therefore the lower bound for the Kobayashi metric and the Sibony metric follows from Corollary \ref{cor}, Theorem \ref{332}, Theorem \ref{645}, Lemma \ref{427}, and Lemma \ref{426}. The lower bound for the Carath\'eodory metric is trivial. 

To show the upper bound of the metrics, we look at the slice of the domain. Let $\re z_1$ be the real normal direction of $\p\W$ at $P$ and $z_2$ be the pseudoconcave direction of $\p\W$ at $P$. Letting $P=0$, we may assume
$$
\set{\re z_1-\abs{z_2}^2 +C\abs{z_1}^2<0}\cap\set{z'=0}\cap U\subset \W\cap U,\quad z'=(z_3,\dots, z_n),
$$
for a sufficiently small neighborhood $U$ of $P$. 
The upper bounds for the Kobayashi and Sibony metrics follow by Corollary \ref{cor}, Theorem \ref{645}, Theorem \ref{332}, Lemma \ref{427}, and Lemma \ref{426}. 


By Hartogs' extension phenomenon, any holomorphic function on $\W$ can be extended to its holomorphic convex hull, hence to a fixed neighborhood
$\W\cup \B(0,r)$

Therefore $F_C^{\W}(P_\dd,\nu)\le F_C^{\W\cup\B(0,r)}(P_\dd,\nu)\approx 1$. 


\end{proof}

\bigskip

\noindent John Erik Forn\ae ss\\
Mathematics Department\\
The University of Michigan\\
East Hall, Ann Arbor, MI 48109\\
USA\\
fornaess@umich.edu\\

\noindent Lina Lee\\
Mathematics Department\\
The University of Michigan\\
East Hall, Ann Arbor, MI 48109\\
USA\\
linalee@umich.edu\\
\end{document}